\documentclass[10pt]{amsart}

\usepackage{graphicx}
\usepackage{amsopn}% for DeclareMathOperator
\usepackage{color}
\usepackage{enumitem}
\usepackage{amsmath}
\usepackage{amssymb}% for \gtrsim
\usepackage{hyperref}% before algorithmic package:
\usepackage{algorithmic}
\usepackage{algorithm}
\usepackage[small,bf]{caption}

\definecolor{disc}{RGB}{255, 0, 0}
\definecolor{alg}{RGB}{255, 191, 0}
\definecolor{diffeqn}{RGB}{128, 255, 0}
\definecolor{funcanal}{RGB}{0, 255, 64}
\definecolor{geotop}{RGB}{0, 255, 255}
\definecolor{probstat}{RGB}{0, 64, 255}
\definecolor{compsciinfo}{RGB}{128, 0, 255}
\definecolor{phys}{RGB}{255, 0, 191}

\theoremstyle{plain}
\newtheorem{theorem}{Theorem}[section]
\newtheorem{lemma}[theorem]{Lemma}
\def\df{\it}

\def\Z{\mathbb{Z}}
\def\smax{s_{\rm max}}
\def\smin{s_{\rm min}}
\def\cv{CV}
\def\revpdf{\mathrm{revpdf}}
\def\revcdf{\mathrm{revcdf}}

\DeclareMathOperator{\Var}{Var}

\title[The \(S\)-metric and preferential attachment]{The \(S\)-metric, the Beichl--Cloteaux approximation, and preferential attachment}

\author[Cory Brunson]{Jason Cory Brunson}

\address{Virginia Bioinformatics Institute, Virginia Tech, Blacksburg, VA 24060}
\email{jabrunso@vbi.vt.edu}

% Actually 2010 MSC

\thanks{Brianna Richardson, Monisha Narayan, Antonio McInness, and Steve Fassino dove into this project with me at the outset, Christopher Franck of LISA has offered facilitative suggestions throughout, and Brian Cloteaux offered helpful comments.
This work grew out of the 2010 REU in Modeling and Simulation in Systems Biology and was partially supported by NSF Award:477855.}

\begin{document}

\begin{abstract}
This paper presents an efficient algorithm to approximate the \(S\)-metric, which normalizes a graph's assortativity by its maximum possible value.
The algorithm is used to track in detail the assortative structure of growing preferential attachment trees, and to study the evolving structure and preferential attachment of several mathematics coauthorship graphs.
These graphs' behavior belies expectations; implications are discussed.

%The \(S\)-metric has grown popular in network studies, as a measure of ``scale-freeness'' restricted to the collection \(G(D)\) of connected graphs with a common degree sequence \(D=(d_1,\ldots,d_n)\).
%The calculation of \(S\) depends on \(\smax\), the greatest value of \(s(g)=\sum_{(i,j)\in E(g)}d_id_j\) among graphs \(g\in G(D)\).
%The original method for computing \(\smax\) is a heuristic construction of a graph \(g^\ast\) with \(s(g^\ast)=\smax\).
%The approximation by Beichl and Cloteaux involves constructing a possibly disconnected graph \(g'\) with \(s(g')\geq s(g^\ast)\) and requires \(O(n^2)\) tests for the graphicality of a degree sequence.
%The present paper uses the Tripathi--Vijay test to streamline this approximation, and thereby to investigate two collections of graphs: Barab\'asi--Albert trees and coauthorship graphs of mathematical sciences researchers.
%Long-term trends in the coauthorship graphs are discussed, and contextualized by insights derived from the BA trees.
%It is known that greater degree-based preferential attachment produces greater variance in degree sequences, and these trees exhibited \(s\)-values restricted to a narrow band.
%In contrast, variance in degree rose over time in the coauthorship graphs in spite of weakening degree-based preferential attachment.
%These observations and their implications are discussed and avenues of future work are suggested.

\end{abstract}

\maketitle

\section{Background}

In the context of graphs, the term ``scale-free'' originally referred to the property of having a scaling degree sequence, but came over time to connote further properties typically associated with high degree--degree correlations, such as self-similarity, that are not entailed by such a degree sequence \cite{ladw-towards}.
In 2005 Li et al \cite{ladw-towards} reconstituted the term in an attempt to reign in this ambiguity.
In the simplest case, let \(G(D)\) denote the collection of connected, undirected graphs having degree sequence \(D=(d_1\geq\cdots\geq d_n)\).
It is generally assumed that \(D\) is scaling, i.e. that
\[d_k\approx Ck^{-\alpha}\text{, or}\ k\approx C{d_k}^{-\alpha},\]
for some constants \(C\) and \(\alpha\).
Li et al introduced the {\df \(s\)-metric}
\[s(g)=\sum_{(i,j)\in E(g)}d_id_j\]
on graphs \(g\in G(D)\), where \(E(g)\) consists of the edges of \(g\).
To position a graph within \(G(D)\) by its \(s\)-metric, they posited graphs \(g_\ast\) and \(g^\ast\) in \(G(D)\), not necessarily unique, of minimal and maximal \(s\)-metrics \(\smin\) and \(\smax\), respectively.
They used the {\df \(S\)-metric}
\[\displaystyle S(g)=s(g)/\smax\]
to draw comparisons among graphs of different sizes and degree sequences, which Li \cite{l-topologies} normalized to
\[S(g)=\frac{s(g)-\smin}{\smax-\smin}\text.\]
It was shown in \cite{ladw-towards} that \(s:G(D)\to\Z_{\geq0}\) is related to Newman's \cite{n-mixing} normalized degree correlation coefficient \(r:G(D)\to[-1,1]\) by a scalar factor depending on \(G(D)\), suggesting the interpretation
\[S(g)=\frac{r(g)-r(g_\ast)}{r(g^\ast)-r(g_\ast)}\]
of the \(S\)-metric as a normalization of \(r\) over \(G(D)\).
%{\color{red} Among the properties exhibited by graphs with higher \(S\)-metrics are assortative mixing, a highly-connected core, and self-similarity.}

Li et al adopted the arithmetic approximation
\begin{equation}\label{sminapprox}
\smin\gtrsim\frac12\sum_{i=1}^nd_id_{n+1-i}
\end{equation}
but constructed \(g^\ast\) in order to compute \(\smax\).
Their construction is not guaranteed to terminate, but if it does then the resulting graph is guaranteed to be \(g^\ast\).
Beichl and Cloteaux \cite{bc-generating} introduced a deterministic, terminating algorithm to construct a graph \(g'\) approximating \(s(g')\gtrsim\smax\).
Their \(g'\) has degree sequence \(D\) but is not necessarily connected; trials suggest that \(\lim_{n\to\infty}{(s(g')-\smax)}/{\smax}=0\) for power law degree sequences.
The algorithm proceeds from scratch, hence requires \(O(n^2)\) iterations of a test for the graphicality of the degree sequence.

A popular test for graphicality is due to Erd\H{o}s and Gallai \cite{eg-graphs}.
It states that there exists a graph \(g\) with degree sequence \(D\) if and only if
\begin{equation}\label{graphicality1}
\sum_{i=1}^{k}d_i\leq k(k-1)+\sum_{i=k+1}^{n}\min(k,d_i)
\end{equation}
across \(1\leq k\leq n-1\).
Tripathi and Vijay \cite{tv-note} proved that it is enough to perform the test once for each distinct number in \(D\):
Let \(\Sigma=(\sigma_1\leq\cdots\leq\sigma_d)\) be a cumulative histogram for \(D\), beginning with the number \(\sigma_1\) of instances of top degree, so that \(d=\max(D)\).
If \(D=(d_1\geq\cdots\geq d_n)\) then \(d_{\sigma_j}\geq d-j+1>d_{\sigma_j+1}\) for each \(j\) (allowing for \(\sigma_j=\sigma_{j+1}\) if \(j\notin D\)).
Then we need only check (\ref{graphicality1}) for \(k=\sigma_j\).
That is, \(D\) is graphical if and only if
\begin{equation}\label{graphicality2}
\sum_{i=1}^{j}(\sigma_i-\sigma_{i-1})(d-i+1)\leq\sigma_j(\sigma_j-1)+\sum_{i=j+1}^d(\sigma_i-\sigma_{i-1})\min(\sigma_j,d-i+1)
\end{equation}
across \(1\leq j\leq d\), where we take \(\sigma_0\) to be zero.
In this note we take advantage of this test to implement a rapid construction of \(g'\).

\section{An improved algorithm}

The Beichl--Cloteaux construction (BC) begins with \(D=D(g)\) and an empty graph \(h\) on \(n\) nodes.
Each step of the construction adds an edge to \(h\) corresponding to a pair of entries of \(D\) that are then reduced by \(1\) each.
Therefore, as \(h\) grows from the empty graph to \(g'\), \(D\) decays from \(D(g)\) to a sequence of zeros.
While Erd\H os--Gallai achieves runtime \(O(d)\) \cite{ilms-quick}, it requires that \(D\) be sorted, a condition violated by BC.
Repairing this by a linear sorting algorithm achieves a runtime \(O(n^3)\) for BC.

The present algorithm instead achieves runtime \(O(n^2d)\).
Among graphs with degree sequences \(D\) that satisfy a scaling law \(P(k)\sim k^{-\alpha}\) (\(\alpha\geq1\)), we get \(\max(D)=O(n^{1/(\alpha-1)})\), which implies a runtime of \(O(n^{(2\alpha-1)/(\alpha-1)})\).
The key is an avatar \(\Phi=(\phi_1,\ldots,\phi_d)\) for \(\Sigma\) on which two operations are computationally inexpensive:
\begin{enumerate}[label={(\roman*)}]
\item adjust \(\Phi\) to reflect the addition of an edge to \(h\) (Lemma~\ref{step});
\item test \(\Phi\) for the graphicality of the remaining \(D\) (Lemma~\ref{graphicality}).
\end{enumerate}
The algorithm adjusts \(D\) along the way.

\subsection{Graphicality and the definition of \(\Phi\)}

The left side of (\ref{graphicality2}) simplifies to
\begin{equation}\label{graphicality3}
\sum_{i=1}^{j}(\sigma_i-\sigma_{i-1})(d-i+1) = \sigma_j(d+1) - \sum_{i=1}^{j}(\sigma_i-\sigma_{i-1})i\text.
\end{equation}
The sum on the right side of (\ref{graphicality3}) decomposes into two summands if \(\sigma_j\leq d-j-1\):
\begin{align*}
& \sum_{i=j+1}^d(\sigma_i-\sigma_{i-1})\min(\sigma_j,d-i+1) \\
&= \left\{\begin{array}{ll}
\displaystyle\sum_{i=j+1}^{d-\sigma_j}(\sigma_i-\sigma_{i-1})(d-i+1) + \sum_{i=d-\sigma_j+1}^d(\sigma_i-\sigma_{i-1})\sigma_j
& d-\sigma_j\geq j+1 \\
\displaystyle\sum_{i=j+1}^d(\sigma_i-\sigma_{i-1})\sigma_j
& j+1>d-\sigma_j
\end{array}\right\} \\
&= \left\{\begin{array}{ll}\displaystyle
(\sigma_{d-\sigma_j}-\sigma_j)(d+1) - \sum_{i=j+1}^{d-\sigma_j}(\sigma_i-\sigma_{i-1})i + (\sigma_d-\sigma_{d-\sigma_j})\sigma_j
& d-\sigma_j\geq j+1 \\
(\sigma_d-\sigma_j)\sigma_j
& j+1>d-\sigma_j
\end{array}\right\}
\end{align*}
A sum is understood to vanish if the starting value for \(i\) exceeds the ending value; if \(i\leq0\) then we take \(\sigma_i=0\), and if \(i>d\) then we take \(\sigma_i=\sigma_d\). Reorganizing (\ref{graphicality2}) and combining like terms produces
\begin{equation}\label{graphicality2a}
\sigma_j(\sigma_d-\sigma_{d-\sigma_j}+\sigma_j-1)
+ (\sigma_{d-\sigma_j}-2\sigma_j)(d+1)
+ \sum_{i=1}^{j}(\sigma_i-\sigma_{i-1})i
- \sum_{i=j+1}^{d-\sigma_j}(\sigma_i-\sigma_{i-1})i
\ \geq\ 0
\end{equation}
when \(d-\sigma_j\geq j+1\) and
\begin{equation}\label{graphicality2b}
(\sigma_d-1)\sigma_j
- \sigma_j(d+1)
+ \sum_{i=1}^{j}(\sigma_i-\sigma_{i-1})i
\ \geq\ 0
\end{equation}
when \(j+1>d-\sigma_j\).
This proves the following lemma.
\begin{lemma}\label{graphicality}
Define \(\phi_j\) across \(1\leq j\leq d\) by (\ref{graphicality2a}) if \(d-\sigma_j\geq j+1\) and by (\ref{graphicality2b}) if \(j+1>d-\sigma_j\).
Then \(D\) is graphical if and only if each \(\phi_j\geq0\).
\end{lemma}
Set \(\Phi=(\phi_1,\ldots,\phi_d)\) and call \(\Phi\) {\df graphical} if each \(\phi_j\geq0\).
Instead of generating \(\Phi\) at each step in the algorithm, we will adjust the \(\phi_i\) directly using knowledge of the degrees of the nodes to be linked at each step.

\subsection{\(\Phi\) and the decay of \(D\)}

While Lemma~\ref{graphicality} provides an inexpensive graphicality test for \(\Phi\), the construction of \(\Phi\) from \(D\) is more expensive than the Tripathi--Vijay test for \(D\).
Improving the efficiency of BC requires a rule for updating \(\Phi\) directly to reflect the addition of an edge to \(h\), or more directly the decrease of two entries of \(D\) by one.
Lemma~\ref{step} provides such an update rule for decreasing a single entry of \(D\).
In the lemma, \(\mathbf{1}_P\) denotes the indicator function for \(P\).

\begin{lemma}\label{step}
Take \(D=(d_1,\ldots,d_n)\) to be a (not necessarily nonincreasing) degree sequence and define \(\Sigma=(\sigma_1\leq\cdots\leq\sigma_d)\) and \(\Phi=(\phi_1,\ldots,\phi_d)\) from \(D\) as above.
Pick \(i\) with \(d_i=k\), set \(D'=(d_1,\ldots,d_{k-1},d_k-1,d_{k+1},\ldots,d_n)\), and define \(\Sigma'\) and \(\Phi'\) from \(D'\) as above.
Across \(1\leq j\leq d\) let
\[\Lambda_j=\Bigg\{\begin{array}{ll}
-\mathbf{1}_{k+\sigma_j>d} & j<k \\
(d-k+1) - 2(\sigma_j-1) + \min(\sigma_k-1,d-k) - \max(0,\sigma_{d-\sigma_k+1}-\sigma_k) & j=k \\
1 & j>k
\end{array}\]
and set \(\Lambda=(\lambda,\ldots,\lambda)\).
Then
\begin{align*}
\Sigma' &= (\sigma_1\leq\cdots\leq\sigma_k-1\leq\sigma_d)\ \ \text{and} \\
\Phi' &= \Phi+\Lambda\text.
\end{align*}
\end{lemma}

\begin{proof}
That \(\sigma'_k=\sigma_k-1\) and \(\sigma'_j=\sigma_j\) for \(j\neq k\) follows from the definition of \(\Sigma\). We further need to show that
\begin{enumerate}[label=({\alph*})]
\item\label{T->k}
if \(j>k\) then \(\phi'_j=\phi_j+1\);
\item\label{T-k}
\(\phi'_k=\phi_k+(d-k+1)-2(\sigma_k-1) +\min(\sigma_k-1,d-k)-\max(0,\sigma_{d-\sigma_k+1}-\sigma_k)\); and
\item\label{T-<k}
if \(j<k\) then \(\phi'_j=\phi_j-\mathbf{1}_{k+\sigma_j>d}\).
\end{enumerate}
We will derive the changes to the terms of (\ref{graphicality1}) that follow from the substitution \(\sigma'_k=\sigma_k-1\) and put them together to obtain the substitutions above.

The left sum in (\ref{graphicality1}) at \(j=k\) becomes
\begin{align*}
\sigma'_k(d+1)-\sum_{i=1}^{k}(\sigma'_i-\sigma'_{i-1})i
&= (\sigma_k-1)(d+1)-\sum_{i=1}^{k-1}(\sigma_i-\sigma_{i-1})i-((\sigma_k-1)-\sigma_{k-1})k \\
&= \sigma_k(d+1)-\sum_{i=1}^{k}(\sigma_i-\sigma_{i-1})i-(d-k+1)\text{,}
\end{align*}
the rightmost term being the change from the original. Since we moved the left sum to the other side of the inequality to produce \(\Phi\), the difference becomes the (positive) first term in \ref{T-k}. For \(j<k\) there is no change, and for \(j>k\) we get
\begin{align*}
\sigma'_j(d+1)-\sum_{i=1}^{j}(\sigma'_i-\sigma'_{i-1})i
&= \sigma_j(d+1)-\sum_{i\neq k,k+1}(\sigma_i-\sigma_{i-1})i \\
&\phantom{=}\ -((\sigma_k-1)-\sigma_{k-1})k-(\sigma_{k+1}-(\sigma_k-1))(k+1) \\
&= \sigma_j(d+1)-\sum_{i=1}^{j}(\sigma_i-\sigma_{i-1})i-1\text{,}
\end{align*}
providing \ref{T->k}.

The first summand on the right of the inequality (\ref{graphicality2}) remains unchanged except at \(j=k\), where
\[\sigma'_k(\sigma'_k-1)=(\sigma_k-1)(\sigma_k-2)=\sigma_k(\sigma_k-1)-2(\sigma_k-1)\text{,}\]
which provides the second summand in \ref{T-k}. This term remains unchanged for \(j>k\). For \(j<k\) we get \(\sigma'_k-\sigma'_{k-1}=\sigma_k-\sigma_{k-1}-1\) and \(\sigma'_{k+1}-\sigma'_k=\sigma_{k+1}-\sigma_k+1\) while otherwise \(\sigma'_i-\sigma'_{i-1}=\sigma_i-\sigma_{i-1}\), so the sum itself may only change if \(\min(\sigma_j,d-k+1)\neq\min(\sigma_j,d-k)\), i.e. when \(k+\sigma_j>d\). In this case the change is \((d-k)-(d-k+1)=-1\). This provides \ref{T-<k}.

At \(j=k\) we only notice the change \(\sigma'_{k+1}-\sigma'_k=\sigma_{k+1}-\sigma_k+1\) among those for \(j>k\), but we also get \(\min(\sigma'_k,d-i+1)=\min(\sigma_k-1,d-i+1)\), which matters when \(\sigma_k=d-i+1\), i.e. \(i=d-\sigma_k+1\). Together, these produce
\begin{align*}
& \sum_{i=k+1}^d(\sigma'_i-\sigma'_{i-1})\min(\sigma'_k,d-i+1) \\
&= (\sigma_{k+1}-\sigma_k+1)\min(\sigma_k-1,d-k)
+ \sum_{i=k+2}^d(\sigma_i-\sigma_{i-1})\min(\sigma_k-1,d-i+1) \\
&= (\sigma_{k+1}-\sigma_k)\min(\sigma_k-1,d-k)
+ \min(\sigma_k-1,d-k) \\
&\phantom{=} + \sum_{i=k+2}^d(\sigma_i-\sigma_{i-1})\min(\sigma_k-1,d-i+1) \\
&= \sum_{i=k+1}^d(\sigma_i-\sigma_{i-1})\min(\sigma_k-1,d-i+1)
+ \min(\sigma_k-1,d-k) \\
&= \mathop{\sum_{i=k+1}^d}_{i>d-\sigma_k+1}(\sigma_i-\sigma_{i-1})\min(\sigma_k,d-i+1) \\
&\phantom{=} + \mathop{\sum_{i=k+1}^d}_{i\leq d-\sigma_k+1}(\sigma_i-\sigma_{i-1})\min(\sigma_k-1,d-i+1)
+ \min(\sigma_k-1,d-k) \\
&= \mathop{\sum_{i=k+1}^d}_{i>d-\sigma_k+1}(\sigma_i-\sigma_{i-1})(d-i+1)
+ \mathop{\sum_{i=k+1}^d}_{i\leq d-\sigma_k+1}(\sigma_i-\sigma_{i-1})(\sigma_k-1) \\
&\phantom{=} + \min(\sigma_k-1,d-k) \\
&= \mathop{\sum_{i=k+1}^d}_{i>d-\sigma_k+1}(\sigma_i-\sigma_{i-1})(d-i+1)
+ \mathop{\sum_{i=k+1}^d}_{i\leq d-\sigma_k+1}(\sigma_i-\sigma_{i-1})\sigma_k
- \mathop{\sum_{i=k+1}^d}_{i\leq d-\sigma_k+1}(\sigma_i-\sigma_{i-1}) \\
&\phantom{=} + \min(\sigma_k-1,d-k) \\
&= \sum_{i=k+1}^d(\sigma_i-\sigma_{i-1})\min(\sigma_k,d-i+1)
- \mathop{\sum_{i=k+1}^d}_{i\leq d-\sigma_k+1}(\sigma_i-\sigma_{i-1})
+ \min(\sigma_k-1,d-k) \\
&= \sum_{i=k+1}^d(\sigma_i-\sigma_{i-1})\min(\sigma_k,d-i+1)
- \max(0,\sigma_{d-\sigma_k+1}-\sigma_k)
+ \min(\sigma_k-1,d-k)\text{,}
\end{align*}
again interpreting \(\sigma_i=0\) for \(i<1\) and \(\sigma_i=\sigma_d\) for \(i>d\), which provides the remaining terms of \ref{T-k}.
\end{proof}

\subsection{The algorithm}

We invoke the reverse cumulative density function \(\revcdf:D\mapsto\Sigma\) and the function \(\Theta\) from \cite{bc-generating} defined by \(\Theta(D,n_1,n_2)=D-e_{n_1}-e_{n_2}\) --- that is, by reducing by one the entries of \(D\) in the \({n_1}^\text{th}\) and \({n_2}^{\text{th}}\) positions.
Algorithm~\ref{drop} (\(\mathrm{Drop}\)) implements the update mechanism of Lemma~\ref{step} and Algorithm~\ref{BCD} (\(\mathrm{BCD}\)) incorporates \(\mathrm{Drop}\) and Lemma~\ref{graphicality} into the construction of \(g'\).

\begin{algorithm}
\caption{\(\mathrm{Drop}(\Sigma,\Phi,m)\): Update \(\Sigma\) and \(\Phi\) to reflect decreasing the degree of one node from \(m\) to \(m-1\).}
\label{drop}
\algsetup{indent = 2em}
\begin{algorithmic}[1]
\REQUIRE \((\Sigma=(\sigma_1,\ldots,\sigma_d),\Phi=(\phi_1,\ldots,\phi_d),m)\)
\ENSURE \(1\leq m\leq d\)
\STATE \(\ell \leftarrow \min(d - \sigma_m + 1, d)\)
\FOR{\(j=\min(\text{index \(j\) such that \(\sigma_j>d-m\)}),\ldots,m-1\)}
\STATE \(\phi_j \leftarrow \phi_j - 1\)
\ENDFOR
\STATE \(\phi_m \leftarrow \phi_m + d - m + 1 - 2(\sigma_m - 1) + \min(\sigma_m - 1, d - m)\)
\IF{\(m < \ell\)}
\STATE \(\phi_m \leftarrow \phi_m - (\sigma_\ell - \sigma_m)\)
\ENDIF
\FOR{\(j=m+1,\ldots,d\)}
\STATE \(\phi_j \leftarrow \phi_j + 1\)
\ENDFOR
\STATE \(\sigma_m \leftarrow \sigma_m - 1\)
\RETURN \((\Sigma,\Phi)\)
\end{algorithmic}
\end{algorithm}

\begin{algorithm}
\caption{\(\mathrm{BCD}(D)\): Construct an edge list for \(g'\) from a graphical degree sequence \(D\).}
\label{BCD}
\algsetup{indent = 2em}
\begin{algorithmic}[1]
\STATE \(n \leftarrow |D|\)
\STATE \(d \leftarrow \max(D)\)
%\STATE \(m \leftarrow \revpdf(D)\)
\STATE \(\Sigma \leftarrow \revcdf(D)\)
\FOR{\(i=1,\ldots,d\)}
\STATE \(\phi_i \leftarrow \sigma_i(\sigma_i - 1) - \sum_{j=1}^{i}(\sigma_j-\sigma_{j-1})(d-j+1) + \sum_{j=i+1}^{d}(\sigma_j-\sigma_{j-1})\min(\sigma_i,d-j+1)\)
\ENDFOR
\STATE \(\Phi \leftarrow (\phi_1,\ldots,\phi_d)\)
\STATE \(E \leftarrow \{\}\)
\STATE \(n_1 \leftarrow 1\)
\STATE \(n_2 \leftarrow 2\)
\WHILE{\(d_{n_1},\ldots,d_n>0\)}
\WHILE{\(d_{n_1}>0\)}
\STATE \((\Sigma',\Phi') \leftarrow \mathrm{Drop}(\Sigma,\Phi,d - d_{n_1} + 1)\)
\STATE \((\Sigma',\Phi') \leftarrow \mathrm{Drop}(\Sigma',\Phi',d - d_{n_2} + 1)\)
\IF{\(\phi_1,\ldots,\phi_d\geq0\)}
\STATE \(E \leftarrow E\cup\{(n_1,n_2)\}\)
\STATE \(D \leftarrow \Theta(D,n_1,n_2)\)
%\FOR{\(k=d - d_{n_1} + 1,d - d_{n_2} + 1\)}
%\STATE \(m_k \leftarrow m_k - 1\)
%\IF{\(k<d\)}
%\STATE \(m_{k+1} \leftarrow m_{k+1} + 1\)
%\ENDIF
%\ENDFOR
\STATE \((\Sigma,\Phi) \leftarrow (\Sigma',\Phi')\)
\ENDIF
\ENDWHILE
\STATE \(n_2 \leftarrow n_2 + 1\)
\ENDWHILE
\end{algorithmic}
\end{algorithm}

\section{Values of \(s\) for simulated and empirical networks}

\subsection{Range of \(s\)-values across trees}

With a (more) rapid algorithm at our disposal, we examine the range of values of \(s(g)\) across \(G(D)\).
We first construct trees using the Barab\'asi--Albert (BA) preferential attachment model \cite{ba-emergence}, which we describe in detail in the next subsection.
The trees contain \(2^p\) nodes each, where \(3\leq p\leq14\).
We specifically consider the relationship between the coefficient of variation in \(D\), given by \(CV=\frac{\sqrt{\Var(D)}}{E(D)}\), and the ratio \(s(g^\ast)/s(g_\ast)\), in order to draw comparisons to \cite{l-topologies} Fig.~4.2.
We approximate \(\smin\) using (\ref{sminapprox}).

\begin{figure}
%\centerline{\includegraphics[width = \columnwidth]{ratio-logCV-BAgap}}
\centerline{\includegraphics[width = \columnwidth]{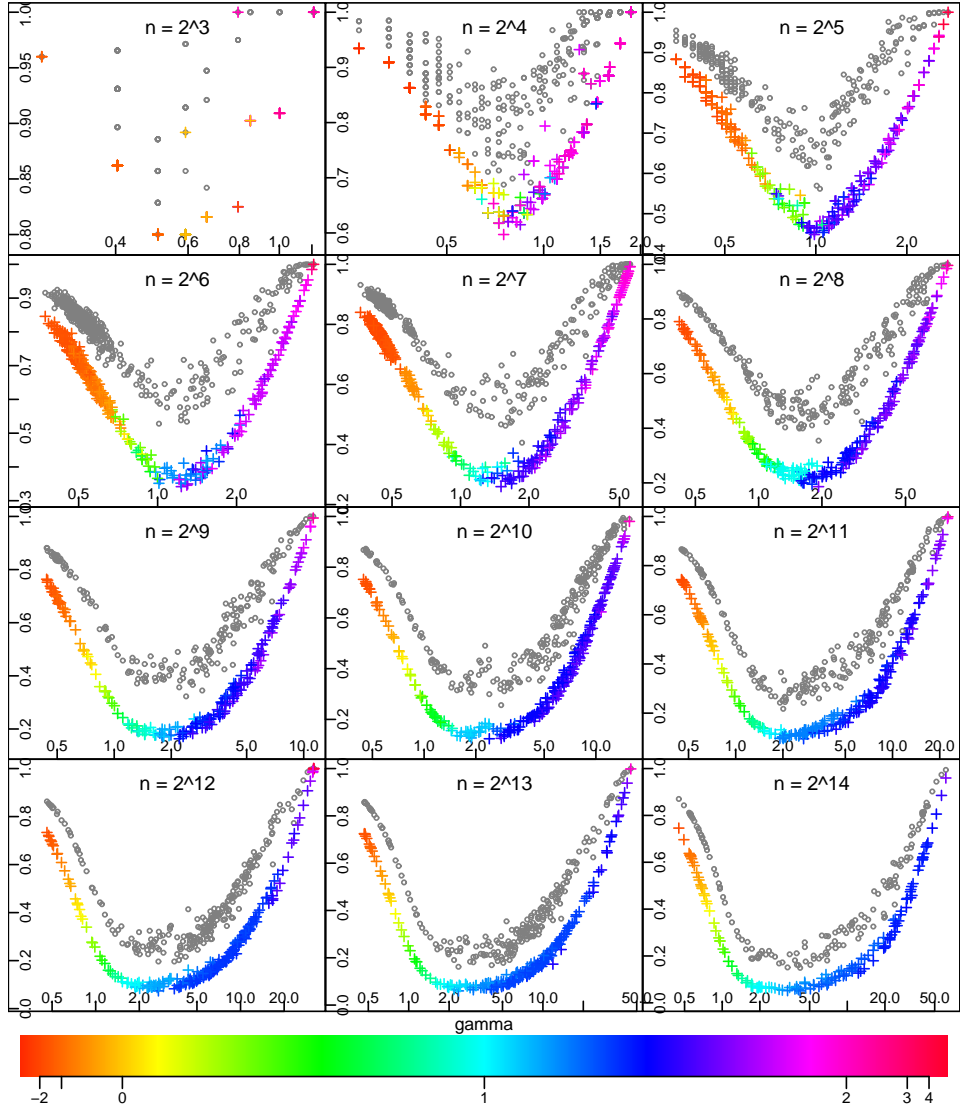}}
%\caption{Data and specifications used in Fig.~\ref{ratioCVBA} on a logarithmic horizontal scale.\label{ratiologCVBA}}
\caption{Plots of the ratio \(s(g_\ast)/s(g^\ast)\) (colored \(+\)'s) and the \(S\)-metric \(s(g)/s(g_\ast)\) (grey \(\circ\)'s) against the coefficient of variation \(\cv(D)\) across a sampling of trees, scaled logarithmically on the horizontal axis.
\(+\)'s are colored by \(\gamma\) according to the color bar, which follows the Riemann projection of \(\mathbb{R}\) onto a circle of radius \(\frac14\) resting at \(\gamma=1\).
We sampled with feedback to fill gaps near the center of each \(\cv\) range.\label{ratioCVBA}}
\end{figure}

Fig.~\ref{ratioCVBA} plots, for each of 100 trees of each size, the \(s\)-metric of each tree normalized by \(\smax\) for its degree sequence (grey \(\circ\)'s) together with \(\smin\) similarly scaled (colored \(+\)'s).
While \(\smin=\smax\) at both ends of the possible \(CV\) range (chains, being the unique graphs with degree sequence \(D=(2,\cdots,2,1,1)\); and stars, ditto \(D=(n-1,1,\ldots,1)\)), stars are easily achieved using preferential attachment with large exponent \(\gamma\), but chains are rare for comparatively large negative \(\gamma\). This is to be expected:
The probability of achieving a star with \(\gamma>0\) is
\[\prod_{i=3}^{n-1}\frac1{1+(i-1)^{1-\gamma}}\text,\]
while that of achieving a chain with \(\gamma<0\) is
\[\prod_{i=3}^{n-1}\frac1{1+(i-2)2^{-1-\gamma}}\text,\]
which converges much more slowly despite both limits being \(1\) as \(\gamma\to\infty\).

The plots are logarithmically-scaled on the \(x\)-axis; to-scale plots of this data reveal what appears to be an asymptotically linear relationship between \(\smin/\smax\) and \(CV\) away from zero.
The logarithmic \(x\)-axis reveals more clearly the impact of \(\gamma\) on \(CV\) and \(s\), specifically that \(\gamma\) and \(CV\) increase together while exponents \(\gamma\approx1\) tend to minimize both \(s(g)/\smax\) and \(\smin/\smax\), i.e.\ produce the greatest diversity of possible \(S\)-values.
This occurs within a narrow range of \(\gamma\) values.

\subsection{Range of \(s\)-values across collaboration networks}

In a forthcoming paper \cite{bfmnrfil-evolutionary}, we examine the collaboration network of mathematicians using the {\df Mathematical Reviews} ({\it MR}) database across 1985--2009.
We study the coauthorship graph across 5-year intervals, as nodes appear and disappear according as corresponding authors begin and cease publishing.
In addition to the aggregate network, we construct and trace coauthorship graphs for a ``pure'' and an ``applied'' network formed from splitting the AMS subject classifications in half (between 58 and 60), and for other, subject-specific subnetworks.

\begin{figure}
\centerline{\includegraphics[width = \columnwidth]{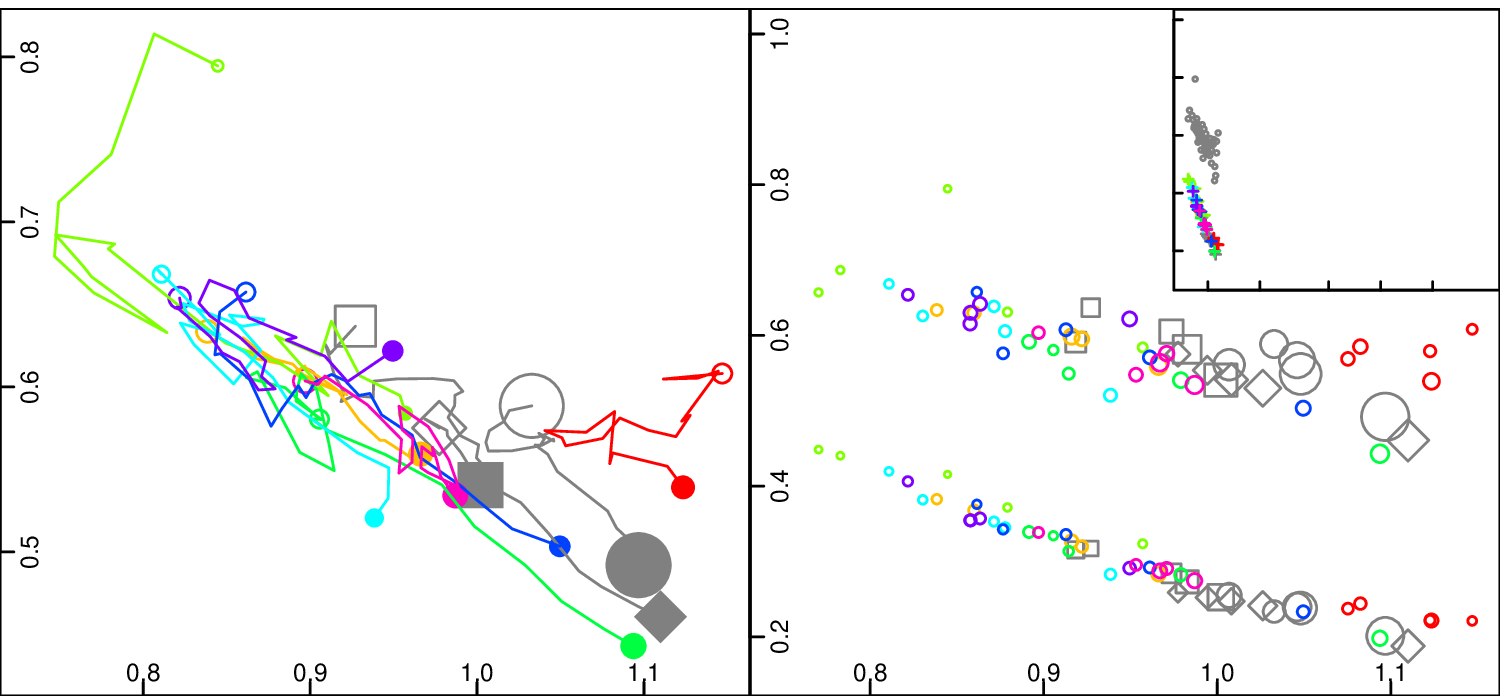}}
\caption{(a) Plot of \(s/\smax\) against the coefficient of variation \(\cv(D)\) with points for the first (open) and last (filled) 5-year intervals and segments connecting the 20 pairs of values for sequential (overlapping) intervals between them:
the aggregate MR network, the pure and applied subnetworks (circles, squares, and diamonds, respectively; all grey), and networks constructed from eight subdisciplines.
(b) Plot of \(s/\smax\) and \(\smin/\smax\) for the same (sub)networks along nonoverlapping 5-year windows across.
The insert depicts \(\smin/\smax\) in a plot with logarithmic \(x\)-axis, \(x\) ranging up to \(\sqrt{\hat{n}}/2\) with \(\hat{n}\) the geometric mean of the sizes of the graphs, and \(y\) ranging from \(.1\) to \(1\), to draw comparisons to Fig.~\ref{ratioCVBA}.\label{ratioCVMR}}
\end{figure}

Fig.~\ref{ratioCVMR} involves eight subdisciplines in addition to the aggregate, pure, and applied {\it MR} networks, determined by their associated MSC classifications and distinguished by color as follows: discrete ({\color{disc}\(\bullet\)}; 03--06), algebra ({\color{alg}\(\bullet\)}; 08--22), differential equations ({\color{diffeqn}\(\bullet\)}; 34--35), functional analysis ({\color{funcanal}\(\bullet\)}; 40--47), geometry and topology ({\color{geotop}\(\bullet\)}; 51--58), probability and statistics ({\color{probstat}\(\bullet\)}; 60--62), computer science and information ({\color{compsciinfo}\(\bullet\)}; 68, 94), and classical physics ({\color{phys}\(\bullet\)}; 70--86).
Fig.~\ref{ratioCVMR}~(a) traces the \(s\)-values of these coauthorship graphs across five adjacent 5-year intervals, along with the minimum \(s\)-values possible under the same degree sequence.
All values are normalized by the maximum \(s\)-value for the degree sequence.
The ordered pairs \((\cv(D),\smin(D)/\smax(D))\) (the insert) occupy a similar stretch of the parameter space to the BA random trees in Fig.~\ref{ratioCVBA} near \(\gamma=0\).
The coincidence of the \(\smin/\smax\) plots across distinct disciplines and various numbers of authors suggests that the degree distributions of collaboration networks are constrained to a limited range of \(\cv\).
The evident, but weaker, coincidence of the \(s/\smax\) plots suggests that, as with BA trees, the topological structure of collaboration networks is subject to constraints beyond those imposed by degree sequence.

The overall trend of the degree sequences of these coauthorship graphs is toward greater variance (Fig.~\ref{ratioCVMR} (a)).
This corresponds in the case of BA trees to increasing the preferential attachment exponent \(\gamma\), as evident from Fig.~\ref{ratioCVBA}.
The question then arises of whether the observed trends can be explained by strengthened preferential attachment in the evolving network \cite{bjnrsv-evolution}.

\section{Preferential attachment}

The mechanism of preferential attachment is perhaps most simply illustrated by the Barab\'asi--Albert tree:
Beginning with a single node \(v_0\) at time \(t=0\), let time proceed at integer steps.
At the \(n^\text{th}\) step, introduce a ``new'' node \(v_n\), then link \(v_n\) to a previous, ``existing'' node \(v_i\), subject to a probability distribution \(P(v_i)\).
The result is a tree whose topology is largely determined by the distribution \(P(v_i)\).
The constraint \(P(v_i)=\Pi(k_i)\), depending only on the degree \(k_i\) of node \(v_i\), constitutes (degree-based) preferential attachment.
If we let \(n_k\) denote the number of existing nodes having degree \(k\) then the probability that node \(v_n\) links to {\sl some} node of degree \(k\) is given by
\begin{equation}\label{prefatt}
P_k=\sum_{k_i=k}P(v_i)=n_k\Pi(k)\text.
\end{equation}
The Barab\'asi--Albert tree arises out of the probability distribution determined by
\begin{equation}\label{batree}
\Pi(k)\propto k^\gamma\text,
\end{equation}
with the case \(\gamma=1\) called {\it linear preferential attachment}.
Formula (\ref{batree}) describes the most common flavor of preferential attachment in the literature.

\(P_k\) is sensitive to the degree sequence of the existing graph, as shown in (\ref{prefatt}), while even \(\Pi(k)=P_k/n_k\) will differ for fixed \(k\) across graphs of different sizes.
Both are then ill-suited to comparing preferntial attachment mechanisms between different networks or over time as a single network evolves.
For this we want to consider the influence of having degree \(k\) on a node's probability of being linked.
This is measured as the factor \(R_k\) defined by \(P_k=R_k\cdot n_k/n\), where \(n=\sum_kn_k\) \cite{n-clustering,tl-empirical}.
\(R_k\) has been called the ``relative probability'' for degree \(k\), though it is not a probability measure in the sense that \(0\leq R_k\leq 1\); \(R_k\) may take any nonnegative value.
From (\ref{prefatt}) we derive \(R_k=n\Pi(k)\) and \(\sum_{i=1}^nR_{k_i}=n\), hence the average value of \(R_k\) across nodes is, suitably, \(1\).

\subsection{Measuring preferential attachment in evolving networks}

Even if a graph grows according to the preferential attachment mechanism, there are difficulties in determining this and in estimating \(R_k\).
Such a graph changes degree sequence with every new node, and the low number of high-degree nodes (especially under a scaling degree sequence) produces high variability in the frequency with which they form links with new nodes.
All this makes a histogrammic approach problematic, but the method of Jeong et al \cite{jnb-measuring} takes one anyway, mitigating these concerns asymptotically.

Consider a graph \(G(T)\) growing over time \(T\geq T_0\).
Designate times \(T_2>T_1>T_0\) with \(\Delta T=T_2-T_1\ll T_1-T_0\).
Refer to nodes, edges, and the incidences between them (of which there are two for every edge) in \(G(T_2)\) as {\it new} or {\it existing} according as they do not or do appear in \(G(T_1)\).
Let us also adopt from \cite{bjnrsv-evolution,tl-empirical} the terms {\it internal} for new edges between existing nodes and {\it external} for edges linking new nodes to existing nodes.
Take \(m_k\) to be the number of new incidences of existing nodes of degree \(k\) to external edges, so that \(m=\sum_km_k\) is the total number of external edges.
Then estimate
\[R_k=P_k\cdot\frac{n}{n_k}\approx\frac{m_k/m}{n_k/n}\text.\]
Under linear preferential attachment, this method should produce a linear relationship between \(R_k\) and \(k\).
For an evolving network in which more recent collaborations are viewed more favorably, a fixed-length window \([T_0,T_1]\) may be slid across a longer interval of data, as we do here.

Barab\'asi et al \cite{bjnrsv-evolution} call this process {\it external preferential attachment} to distinguish it from the attachment of existing nodes to other existing nodes as a network evolves, which they call {\it internal preferential attachment}.
This phenomenon may be defined and measured using a similar method:
Take \(\Pi(k_1,k_2)\) to be the probability that two unlinked existing nodes of degree \(k_1\) and \(k_2\).
Internal preferential attachment has been modeled similarly to external, as
\[\Pi(k_1,k_2)\propto(k_1k_2)^\gamma\text,\]
which, using the method of \cite{bjnrsv-evolution}, produces a linear locus on a log-log plot of \(R_{k_1,k_2}\) versus \(k_1k_2\).
By a procedure analogous to the external case, we may approximate
\[R_{k_1,k_2}\approx\frac{m_{k_1,k_2}/m}{n_{k_1,k_2}/n}\text,\]
the influence of both degrees on the probability of two nodes becoming linked, with \(n_{k_1,k_2}\), \(n\), \(m_{k_1,k_2}\), and \(m\) defined similarly to their analogs in the external discussion.

\subsection{Preferential attachment in coauthorship graphs}

In the MR network we adopt a 5-year interval for \([T_0,T_1]\) with increments of \(\Delta T=\) 1 year.
The coarse time units of the data (one year) prevent finer increments, and the duration of our data (25 years) limits the length of an ``existing'' interval than can be slid across the duration to produce useful time series data.
Moreover, the decay of expertise implies that only a limited number of years prior to new collaborations will have useful predictive power.

\begin{figure}
\centerline{
\includegraphics[trim = 0cm 0cm 0cm 0cm, clip = true, width = \columnwidth]{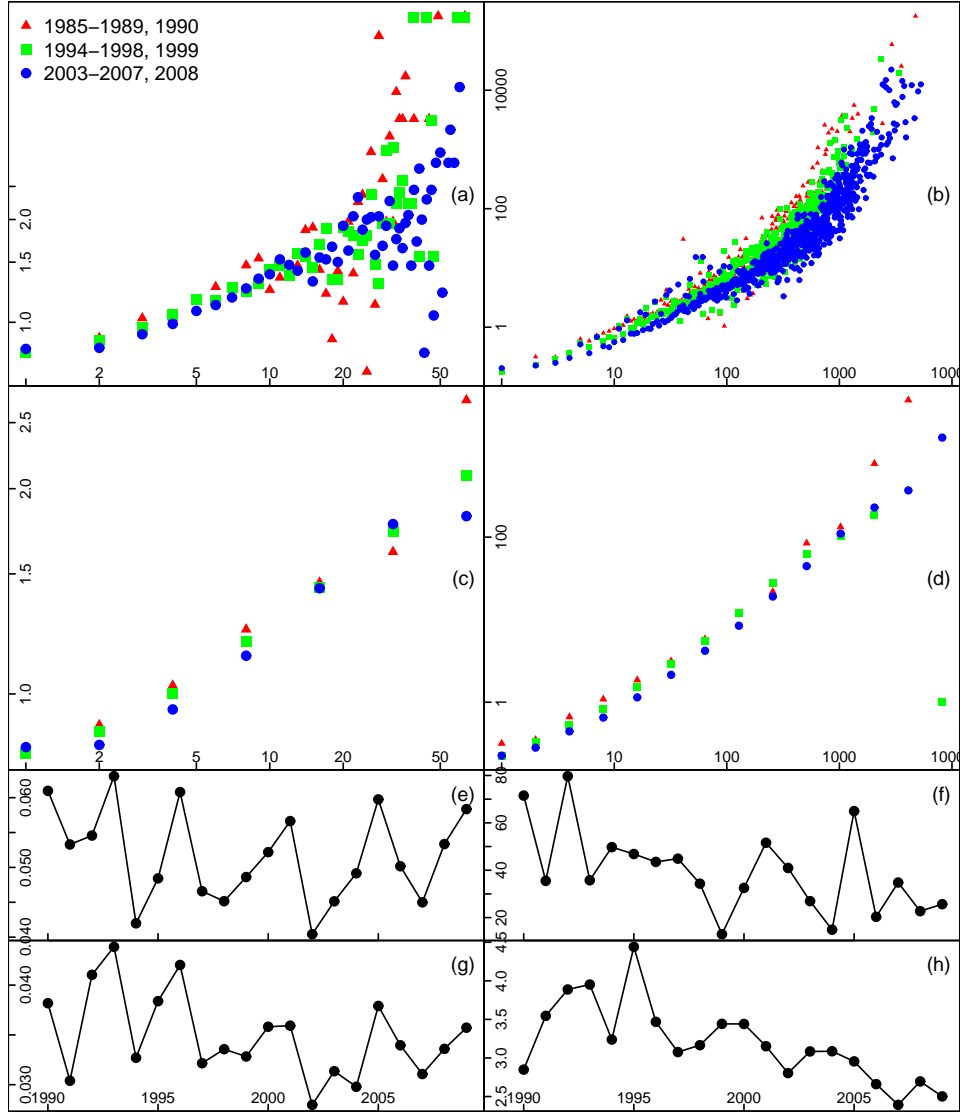}
}
\caption{(a) Plots of the relative probability \(R_k\) of linking by external preferential attachment for three 6-year intervals up to \(k=2^6\).
(b) Plots of \(R_{k_1,k_2}\) for internal preferential attachment for the same intervals.
(c) \(R_k\) calculated from nodes pooled using logarithmic binning.
(d) \(R_{k_1,k_2}\) using logarithmic binning.
(e) Variance of \(R_k\) across nodes.
(f) Variance of \(R_{k_1,k_2}\) across unlinked pairs of nodes.
(g) \(\Var R_k\) restricted to nodes of degree \(\leq 2^4\).
(h) \(\Var R_{k_1,k_2}\) restricted to pairs of nodes with degree product \(\leq 2^9\).
All calculations were done on largest components.\label{prefattMR}}
\end{figure}

Fig.~\ref{prefattMR} (a) plots \(R_k\) against \(k\) across \(1\leq k\leq 64\) for three of these 6-year windows, showing consistent behavior across our interval. In (b) the same three windows show similarly consistent behavior in \(R_{k_1,k_2}\), though with subtle yet perceptible weakening shift.
To reduce visual noise, we pool values of \(k\) (respectively, \(k_1k_2\)) into logarithmic bins \(\{1\}\), \(\{2\}\), \([3,4]\), \([5,8]\), and so on and plotted aggregate ``relative probabilities'' for the collection of nodes (resp., pairs of nodes) associated with each bin.
This also has the effect of removing the false impression that the \(R_{k_1,k_2}\)--\(k_1k_2\) relationship steepens after about \(k_1k_2=200\), which is an artifact of the decreasing number of pairs of nodes of any specific product of degrees.
While the trends could be modeled linearly, they exhibit clearly nonlinear behavior, in particular higher values of \(R_k\) (resp. \(R_{k_1,k_2}\)) for low-degree nodes (pairs) than would be predicted by extrapolating from the higher-degree nodes or pairs.

Given such data along a sliding window, the question concluding the previous section amounts to whether a suitable measure of the strength of the preferential attachment mechanism increases over time.
It would be possible to compare linked nodes and pairs directly, viewed as samples from weighted distributions (as approximated by relative probabilities), against a null hypothesis of equal weight; but this would not allow us to view trends in the patterns of these weights.
If we assume that the underlying mechanism takes the form (\ref{batree}) then we may use the established trick of fitting a linear curve to the log-log plots of \(R_k\) and \(R_{k_1,k_2}\) or, again to reduce noise, to those of their partial sums \(S_K=\sum_{k=1}^KR_k\) and \(S_K=\sum_{k_1k_2\leq K}R_{k_1,k_2}\), similarly to \cite{bjnrsv-evolution}.
The slope of the best-fit line gives an approximation to \(\gamma\) (respectively, \(\gamma+1\)).
However, the nonlinearity of the log-log plots themselves suggests that the traditional model (\ref{batree}) is inadequate and prompts our use of a different metric.
(We did fit lines to \(S_K\) in both cases to approximate \(\gamma\) over time, cutting data at \(k\leq 32\) and \(k_1k_2\leq 256\). A linear fit to the time series of best-fit exponents produced no evidence of a nonconstant trend.)

Instead, we consider the variance of \(R_k\).
(The internal case is similar to the external, which we cover here.)
By definition, the average value of \(R_k\) across nodes is \(1\).
Let us make the assumption that the distribution \(R_k\) for each year is not fixed but drawn from a family of distributions \(R_k^{(i)}\) related by
\[R_k^{(i)}=(R_k^{(j)})^{\kappa_{ij}}\text,\]
slightly more modest than (\ref{batree}) and including it as a special case.
We may parametrize this family as \((R_k^{(1)})^t\) across \(-\infty<t<\infty\), and if we restrict ourselves to positive \(t\) (which would mean restricting (\ref{batree}) to positive, or negative, \(\gamma\)) then \(\Var(R_k^{(1)})^t\) will increase with \(t\).
Accordingly, we propose to measure the strength of preferential attachment over time by this variance.
To account for the influence of high-degree nodes, we compute time series for external preferential attachment three ways: including all nodes; including only nodes of degree \(\leq 2^6\), and including only nodes of degree \(\leq 2^4\).
For internal preferential attachment, we likewise use no cutoff and cutoffs of \(2^{13}\) and \(2^9\) on the product \(k_1k_2\).
We then fit a linear model to each time series and compute a two-sided \(p\)-value against the null hypothesis that its slope \(m=0\).

The time series with no cutoff and with the lower cutoff in both cases comprise Fig.~\ref{prefattMR} (e--h).
The \(p\)-values arising from each significance test comprise Table~\ref{pvalues}.
Each suggests a downward trend and produces a linear fit with negative slope, and in both cases \(p\)-values shrink as the cutoff decreases.
The significance is substantially greater in the case of internal preferential attachment.

\begin{table}
\centering
\small
\begin{tabular}{rr|rr}
cutoff & external (\(\Var R_k\)) & cutoff & internal (\(\Var R_{k_1,k_2}\)) \\
\hline
none & .47\phantom{0} & none & .015\phantom{00} \\
64 & .35\phantom{0} & 8192 & .0036\phantom{0} \\
16 & .087 & 512 & .00069 \\
\end{tabular}
\caption{Significance (\(p\)-values) of the linear fit to the time series of \(\Var R_k\) and \(\Var R_{k_1,k_2}\), omitting nodes from the calculation at different cutoffs.\label{pvalues}}
\end{table}

\section{Summary and further work}

A rapid approximation to the \(\smax\) graph should allow network researchers to take fuller advantage of the \(S\)-metric to distinguish among large real-world networks, particularly those with approximately scaling degree sequences.
We took the opportunity to view the range of behaviors exhibited by Barab\'asi--Albert trees of up to \(2^{14}\) nodes across a spectrum of exponents.
The \(s\)-values of these BA trees are confined within a narrow band between \(\smin\) and \(\smax\).
Moreover, as variation in degree sequences increases with \(\gamma\), the ratio of smallest to largest \(s\)-values across graphs fixed at each degree sequence reaches a minimum near \(\gamma=1\).
Despite great diversity among the {\it Mathematical Reviews} coauthorship graphs, these networks exhibit a limited range of assortativity given their degree sequence, which is similar to the band attained by BA trees.
There has been a long-term decline in \(S\), coupled with rising variance in degree sequence, that is not explained by changes in degree-based preferential attachment.
This mechanism has, if anything, been weakening in recent years.
Further investigation into trends in \(S\) over time and possible real-world factors may provide greater insight into the mechanisms of real-world network evolution and a more robust understanding of scale-freeness itself.

\bibliographystyle{plain}
\bibliography{s_max}

\end{document}